\documentclass[11pt]{amsart}

\usepackage[all]{xy}
\usepackage{amssymb,xspace}
\usepackage{graphicx}

\newcommand*{\C}{\ensuremath{\mathbb C}\xspace}
\newcommand*{\E}{\ensuremath{\mathcal E}\xspace}
\newcommand*{\F}{\ensuremath{\mathcal F}\xspace}
\newcommand*{\Ff}{\ensuremath{\mathbb F}\xspace}
\newcommand*{\Ic}{\ensuremath{\mathcal I}\xspace}
\newcommand*{\M}{\ensuremath{\mathcal M}\xspace}
\newcommand*{\Nc}{\ensuremath{\mathcal N}\xspace}
\newcommand*{\Oo}{\ensuremath{\mathcal O}\xspace}
\newcommand*{\PP}{\ensuremath{\mathbb P}\xspace}
\newcommand*{\Z}{\ensuremath{\mathbb Z}\xspace}

\DeclareMathOperator{\bd}{bd}
\DeclareMathOperator{\codim}{codim}

\DeclareMathOperator{\trdeg}{tr.deg}

\let\ph\varphi

\newcommand*{\lst}[3][1]{\ensuremath{#2_{#1}, \ldots, #2_{#3}}\xspace}

\newtheorem{proposition}{Proposition}[section]

\newtheorem{lemma}[proposition]{Lemma}
\newtheorem{corollary}[proposition]{Corollary}

\theoremstyle{remark}
\newtheorem{note}[proposition]{Remark}

\theoremstyle{definition}
\newtheorem{definition}[proposition]{Definition}
\newtheorem{Not}[proposition]{Notation}
\newtheorem{constr}[proposition]{Construction}

\title[Algebraic and
  non-algebraic neighborhoods]{On algebraic and
  non-algebraic neighborhoods of rational curves}

\author{Serge Lvovski}
\address{National Research University Higher School of Economics,
  Moscow, Russia\hfil\break
Federal State Institution "Scientific-Research Institute for System
Analysis of the Russian Academy of Sciences" (SRISA) 
}
\email{lvovski@gmail.com}



\thanks{This study was partially supported by the HSE University Basic
  Research Program and by SRISA research project FNEF-2022-0007
  (Reg. No 1021060909180-7-1.2.1).} 

\begin{document}

\begin{abstract}
We prove that for any $d>0$ there exists an embedding of the Riemann
sphere $\PP^1$ in a smooth complex surface, with
self-intersection~$d$, such that the germ of this embedding cannot be
extended to an embedding in an algebraic surface but the field of
germs of meromorphic functions along $C$ has transcendence degree~$2$
over~\C. We give two different constructions of such neighborhoods,
either as blowdowns of a neighborhood of the smooth plane conic, or as
ramified coverings of a neighborhood of a hyperplane section of a
surface of minimal degree.

The proofs of non-algebraicity of these neighborhoods are based on a
classification, up to isomorphism, of algebraic germs of embeddings of
$\PP^1$, which is also obtained in the paper.
\end{abstract}

\maketitle

\section{Introduction}

In this paper we study germs of embeddings of the Riemann sphere aka
$\PP^1$ in smooth complex surfaces (see precise definitions in
Section~\ref{subsec} below). The structure of such germs for which the
degree of the normal bundle, which is equal to the self-intersection
index of the curve is question, is non-positive, is well known and
simple: if $(C\cdot C)=d\le 0$, then for any such $d$ there exists
only one germ up to isomorphism, and these germs are algebraic (see
\cite{Grauert} for the negative degree case and \cite{Savelyev} for
the zero degree case). Germs of embeddings $(C,U)$, where
$C\cong\PP^1$, $U$ is a smooth complex surface, and $(C\cdot C)>0$,
are way more diverse.

Let us say that a germ of neighborhood of $C\cong\PP^1$ is algebraic
if it is isomorphic to the germ of embedding of $C$ in a smooth
algebraic surface.  M.~Mishustin, in his paper~\cite{Mishustin},
showed that the space of isomorphism classes of germs of embeddings of
$C\cong\PP^1$, $(C\cdot C)>0$, in smooth surfaces, is
infinite-dimensional, so one would expect that ``most'' germs of such
embeddings are not algebraic. In this paper we will
construct explicit examples of non-algebraic germs of embeddings of
$\PP^1$.

An interesting series of such examples was constructed in a recent
paper by M.~Falla Luza and F.~Loray~\cite{FLL2022}. For each $d>0$
they construct an embedding of $C\cong\PP^1$ is a smooth surface such
that $(C\cdot C)=d$ and the field of germs of meromorphic functions
along $C$ consists only of constants. For curves on an algebraic
surface this is impossible, so the germs of neighborhoods constructed
in~\cite{FLL2022} are examples of non-algebraic germs.

In Section~5.3 of the paper~\cite{FLL-MMJ} the same authors give an example of a
non-algebraic germ of neighborhood of $\PP^1$, with
self-intersection~$1$, for which the field of germs of meromorphic
functions is as big as possible: it has transcendence degree~$2$ over~\C.

The aim of this paper is to construct two series of explicit
examples of non-algebraic germs of neighborhoods of $\PP^1$ for which
the field of germs of meromorphic functions is as big as possible. 

The first of these series is constructed in Section~\ref{sec:blowdowns}.
Specifically, for any integer $m\ge 5$ we construct a non-algebraic
germ of neighborhood $(C,U)$ such that $C\cong \PP^1$, $(C\cdot C)=m$,
and one can blow up $m-4$ points on $C$ to obtain the germ of
neighborhood of the conic in the plane. The field of germs of
meromorphic functions along~$C$ has transcendence degree $2$ over~\C
(Construction~\ref{constr1} and Proposition~\ref{prop:constr1}).

The second series of examples is constructed in
Section~\ref{sec:ramified}. For any positive integer $n$ we construct
a non-algebraic germ of neighborhood $(C,U)$ such that $C\cong \PP^1$,
$(C\cdot C)=n$, and the field of germs of meromorphic functions along $C$ 
has transcendence degree $2$ over~\C, as a ramified two-sheeted covering of the germ of a
neighborhood of $\PP^1$ with self-intersection~$2n$. This construction
is a generalization of that from~\cite[Section 5.3]{FLL-MMJ} (and
coincides with the latter for $n=1$), but the method of proof if
non-algebraicity is different. See Construction~\ref{constr2} and
Proposition~\ref{prop:constr2}. 

I do not know whether the germs of neighborhoods constructed in Sections~\ref{sec:blowdowns} and~~\ref{sec:ramified} 
are isomorphic.

The proofs of non-algebraicity of the neighborhoods constructed in
Sections~\ref{sec:blowdowns} and~~\ref{sec:ramified} are based on the
classification of algebraic germs of neighborhoods of $\PP^1$. It
turns out that any such germ with self-intersection $d>0$ is
isomorphic to the germ of a hyperplane section of a surface of
degree~$d$ in $\PP^{d+1}$; moreover, any isomorphism of germs of such
embeddings $(C_1,F_1)$ and $(C_2,F_2)$ (where $F_j$, $j=1,2$, are
surfaces of degree $d$ in $\PP^{d+1}$ and $C_j$ a hyperplane section
of $F_j$) is induced by a linear isomorphism between the surfaces
$F_1,F_2\subset\PP^{d+1}$ (Propositions~\ref{prop:alg.neighborhoods}
and~\ref{prop:rigid}). The key role in the proofs is played by
Lemma~\ref{lemma:extn2}.

The paper is organized as follows. In Section~\ref{recap} we recall
the properties of surfaces of degree $d$ in $\PP^{d+1}$. In
Section~\ref{sec:C.C>0} we obtain a classification of algebraic
neighborhoods of~$\PP^1$. Finally, in Sections~\ref{sec:blowdowns}
and~\ref{sec:ramified}  we construct two series of examples of
non-algebraic neighborhoods.

\subsection*{Acknowledgements}
I am grateful to Frank Loray and Grigory Merzon for useful discussions. 

\subsection{Notation and conventions}\label{subsec}

All algebraic varieties are varieties over~\C. All topological
terminology pertains to the classical (complex) topology.

Suppose we are given the pairs $(C_1,S_1)$ and $(C_2,S_2)$,
where $C_j$, $j=1,2$, are projective algebraic curves contained in
smooth complex analytic surfaces~$S_j$. We will say that these two pairs
are \emph{isomorphic as germs of neighborhoods} if there exists an
isomorphism $\ph\colon V_1\to V_2$, where $C_1\subset V_1\subset S_1$,
$C_2\subset V_2\subset S_2$, $V_1$ and $V_2$ are open, such that
$\ph(C_1)=C_2$. In this paper we are mostly concerned with germs of
neighborhoods, but sometimes, abusing the language, we will write
``neighborhood'' instead of ``germ of neighborhoods''; this should not
lead to confusion.

If $C$ is a projective algebraic curve on a smooth complex analytic
surface~$S$, then a \emph{germ of holomorphic}
(resp.\ \emph{meromorphic}) \emph{function along~$S$} is an
equivalence class of pairs $(U,f)$, where $U$ is a neighborhood of $C$
in $S$, $f$ is a holomorphic (resp.\ meromorphic) function on $U$, and
$(U_1,f_1)\sim (U_2,f_2)$ if there exists a neighborhood $V\supset C$,
$V\subset U_1\cap U_2$, on which $f_1$ and $f_2$ agree. Germs of
holomorphic (resp.\ meromorphic) functions along $C\subset S$ form a
ring (resp.\ a field). The field of germs of meromorphic functions
along~$C\subset S$ will be denoted, following the
paper~\cite{FLL2022}, by~$\M(S,C)$. If the self-intersection index
$(C\cdot C)$ is positive, then, according to Theorem~2.1
from~\cite{Orevkov}, the curve $C\subset S$ has a fundamental system
of pseudoconcave neighborhoods, and it follows from~\cite[Th\'eor\`eme
  5]{Andreotti} that $\trdeg_\C\M(S,C)\le2$ (here and below,
$\trdeg_\C$ means ``transcendence degree over~\C'').

A germ of neighborhood of a projective curve~$C$ will be called
\emph{algebraic} if it is isomorphic, as a germ, to the germ of
neighborhood of $C$ in $X$, where $X\supset C$ is a smooth projective
algebraic surface; since desingularization of algebraic surfaces over
\C exists, one may as well say that a germ of neighborhood of $C$ is
algebraic if it is isomorphic to the germ of a neighborhood $C\subset
X$, where $X$ is a an arbitrary smooth algebraic surface, not
necessarily projective. In this case, $\trdeg_\C\M(X,C)=2$ since this
field contains the field of rational functions on~$X$.

\begin{note}\label{remark}
A germ of neighborhood of a projective curve~$C$ \emph{with positive
  self-intersection} is algebraic if and only if it is isomorphic to
the germ of embedding of~$C$ into a compact complex surface. Indeed,
any smooth complex surface containing a projective curve with positive self-intersection, can be embedded in $\PP^N$ for some~$N$ (see
\cite[Chapter~IV, Theorem 6.2]{BPHVdV}).
\end{note}

A projective subvariety $X\subset\PP^N$ is called
\emph{non-degenerate} if it does not lie in a hyperplane, and
\emph{linearly normal} if the natural homomorphism from
$H^0(\PP^N,\Oo_{\PP^N}(1))$ to $H^0(X,\Oo_X(1))$ is surjective. If $X$
is non-degenerate, the latter condition holds if and only if $X$ is
not an isomorphic projection of a non-degenerate subvariety of
$\PP^{N+1}$.

Non-degenerate projective subvarieties $X_1\subset\PP^{N_1}$ and
$X_2\subset\PP^{N_2}$ will be called \emph{projectively isomorphic} if
$N_1=N_2$ and there exists a linear isomorphism $f\colon
\PP^{N_1}\to\PP^{N_2}$ such that $f(X_1)=X_2$ .

If $C_1$ and $C_2$ are two projective algebraic curves on a smooth
complex surface, then $(C_1\cdot C_2)$ is their intersection index.

The terms ``vector bundle'' and ``locally free sheaf'' will be used
interchangeably. 

If $C$ is a smooth curve on a smooth complex surface~$X$, then
$\Nc_{X|C}$ is the normal bundle to $C$ in~$X$. 

If \E is a vector bundle on an algebraic variety, then $\PP(\E)$ (the
projectivisation of \E) is the algebraic variety such that its points
are lines in the fibers of~\E.

If \F is a coherent sheaf on a complex space~$X$, we will sometimes
write $h^i(\F)$ instead of $\dim H^i(X,\F)$.

By definition, a projective surface $X$ has \emph{rational
  singularities} if $p_*\Oo_{\bar X}=\Oo_X$ and $R^1p_*\Oo_{\bar X}=0$
for some (hence, any) desingularization $p\colon\bar X\to X$.

If $f\colon S\to T$  is a dominant morphism of smooth
projective algebraic surfaces, we will say that its \emph{critical
  locus} is the set
\begin{equation*}
R=f(\{x\in S\colon \text{$df_x$ is degenerate}\})  \subset T,
\end{equation*}
and that its \emph{branch divisor} $B\subset T$ is the union of
one-dimensional components of~$R$.

\section{Recap on surfaces of minimal degree}\label{recap}

In this section we will recall, without proofs, some well-known
results. Most of the details can be found in~\cite{EisHar}.

If $X\subset\PP^N$ is a non-degenerate irreducible projective variety,
then 
\begin{equation}\label{eq:deg>codim}
\deg X\ge \codim X+1,
\end{equation}
and there exists a classification of the varieties for which the
lower bound~\eqref{eq:deg>codim} is attained. We reproduce this
classification for the case~$\dim X=2$.

\begin{Not}
For any integer $n\ge0$, put $\Ff_n=\PP(\Oo_{\PP^1}\oplus
\Oo_{\PP^1}(n))$.    
\end{Not}
The surface $\Ff_0$ is just the quadric $\PP^1\times\PP^1$; if $n>0$,
then the natural projection $\Ff_n\to\PP^1$ has a unique section
$E\subset\Ff_n$ such that $(E\cdot E)=-n$. The section $E$ will be
called \emph{the exceptional section} of~$\Ff_n$. The divisor class
group of $\Ff_n$ is generated by the class~$e$ of the exceptional
section and the class $f$ of the fiber (if $n=0$, we denote by
$e$ and $f$ the classes of lines of two rulings on
$\PP^1\times\PP^1$). One has
\begin{equation}\label{eq:mult.table}
  (e\cdot e)=-n,\quad (e\cdot f)=1,\quad (f\cdot f)=0.
\end{equation}

If $r\geqslant 0$ is an integer and $(n,r)\ne(0,0)$, then the complete
linear system $|e+(n+r)f|$ on $\Ff_n$ has no basepoints and defines a
mapping $\Ff_n\to \PP^{n+2r+1}$.

\begin{Not}
If $n,r$ are non-negative integers and $(n,r)\ne(0,0)$, then
by $F_{n,r}\subset\PP^{n+2r+1}$   we will denote the image of the
mapping $\Ff_n\to\PP^{n+2r+1}$ defined by the complete linear system
$|e+(n+r)f|$. 
\end{Not}

It follows from~\eqref{eq:mult.table} that the variety
$F_{n,r}\subset\PP^{n+2r+1}$ is a surface of degree~$n+2r$. If $r>0$,
the surface $F_{n,r}$ is smooth and isomorphic to $\Ff_n$. The surface
$F_{n,0}$ with $n\ge 2$ is the cone over the normal rational curve of
degree $n$ in $\PP^n$ (this cone is obtained by contracting the exceptional
section on $\Ff_n$), and the surface $F_{1,0}$ is just the plane; the
surfaces $F_{n,0}$ are normal and, moreover, have rational
singularities. If $n>0$ and $r>0$, then the exceptional section on
$F_{n,r}$ is a rational curve of degree~$r$.  Finally, the surface
$F_{0,1}\subset\PP^3$ is the smooth quadric.

Recall that the quadratic Veronese surface $V\subset\PP^5$ is the image of the mapping
$v\colon \PP^2\to\PP^5$ defined by the formula
\begin{equation}\label{eq:Veronese}
	v\colon (x_0:x_1:x_2)\mapsto (x_0^2:x_1^2:x_2^2:x_0x_1:x_0x_2:x_1x_2);
\end{equation}
one has $\deg V=4$. The mapping $v$ induces an isomorphism from $\PP^2$
to~$V$; hyperplane sections of $V$ are images of conics in~$\PP^2$.

\begin{proposition}\label{prop:min.degree}
If $X\subset\PP^N$ is a non-degenerate irreducible projective surface,
then $\deg X\ge N-1$ and the bound is attained if and only if either
$X=F_{n,r}$, where $(n,r)\ne(0,0)$, or $X=V\subset\PP^5$, where $V$ is
the quadratic Veronese surface.  
\end{proposition}

If $(n,r)\ne (n',r')$, then the surfaces $F_{n,r}$ and $F_{n',r'}$ are
\emph{not} projectively isomorphic, and none of them is projectively
isomorphic to the Veronese surface~$V$. (Indeed, if $n\ne n'$ then
$F_{n,r}$ and $F_{n',r'}$ are not isomorphic even as abstract
varieties, and if $n=n'$ but $r\ne r'$ then their degrees are
different; speaking of the Veronese surface, it does not contain lines
while each $F_{n,r}$ is swept by lines.)

Suppose now that $X\subset \PP^{N+1}$, $N\ge 2$, is a surface of
minimal degree (i.e., of degree~$N$) and $p\in F$ is a smooth point
(i.e., either $X\ne F_{n,0}$ and $p$ is arbitrary or $X=F_{n,0}$ and
$p$ is not the vertex of the cone). Let $\pi_p$ denote the projection
from $\PP^{N+1}$ to $\PP^N$ with the center~$p$.

\begin{proposition}\label{prop:proj}
In the above setting, the projection $\pi_p$ induces a birational
mapping from $X$ onto its image. If $X'\subset \PP^N$ is \textup(the
closure of\textup) the image of $\pi_p\colon X\dasharrow \PP^N$,
then $X'$ is also a surface of minimal degree.

If $X=F_{n,r}$, where $n>0$ and $r>0$, then $X'=F_{n+1,r-1}$ if $p$
lies on the exceptional section and $X'=F_{n-1,r}$ otherwise.

If $X=F_{0,r}$, $r>0$, then $X'=F_{1,r-1}$.

If $X=F_{n,0}$, $n\ge 2$, then $X'=F_{n-1,0}$.

Finally, if $X=V\subset\PP^5$, then $X'=F_{1,1}\subset\PP^4$.
\end{proposition}

These projections for surfaces of minimal degree~$\le 6$ are depicted in
Figure~\ref{fig:hierarchy}. Observe that no arrow points at the
surface~$V$; this fact will play a crucial role in the sequel.

\begin{figure}
\[
\xymatrix{
{\text{degree}}\\
6&F_{6,0}\ar[d]&F_{4,1}\ar[dl]\ar[d]&F_{2,2}\ar[dl]\ar[d]&F_{0,3}\ar[dl]\\
5& F_{5,0}\ar[d]&F_{3,1}\ar[dl]\ar[d]&F_{1,2}\ar[dl]\ar[d]\\
4&F_{4,0}\ar[d]&F_{2,1}\ar[dl]\ar[d]&F_{0,2}\ar[dl]&V\ar[dll]\\
3&F_{3,0}\ar[d]&F_{1,1}\ar[dl]\ar[d]\\
2&F_{2,0}\ar[d]&F_{0,1}\ar[dl]\\
1&F_{1,0}
}
\]  
  \caption{Projections of surfaces of minimal degree}\label{fig:hierarchy}
  
\end{figure}
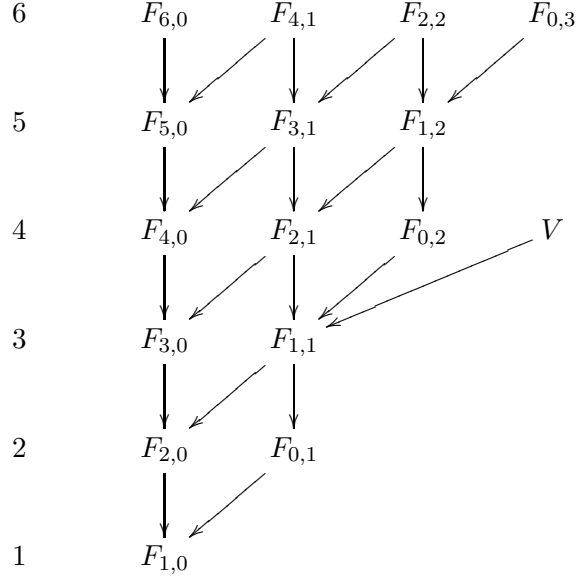

The birational transformations induced by the projections form
Proposition~\ref{prop:proj} can be described explicitly.

\begin{proposition}\label{prop:proj2}
Suppose that $F\subset\PP^N$ is a surface of minimal degree, $p\in F$
is a non-singular point, and $\pi_p\colon F\dasharrow \PP^{N-1}$ is the
projection from~$p$. The projection $\pi_p$ acts on the surface $F$ as follows.

\begin{enumerate}
\item
If $F=F_{n,r}\subset \PP^{n+2r+1}$, where $r>1$, then the projection
$\pi_p$ blows up the point~$p$ and blows down the strict transform of
the only line on~$F$ passing through~$p$.

\item
If $F=F_{n,1}\subset\PP^{n+3}$, where $n>0$ (in this case the
exceptional section of $F\cong \Ff_n$ is a line), then two cases are
possible. 

If the point $p$ does not lie on the exceptional section, then the
projection $\pi_p$ blows up the point~$p$ and blows down the strict
transform of the only line passing through~$p$, and the image of the
projection is the surface $F_{n-1,1}\subset\PP^{n+2}$. If, on the other hand, $p$ lies
on the exceptional section, then the projection $\pi_p$ blows up the
point~$p$ and blows down both the strict transform of the fiber
passing through~$p$ and the strict transform of the exceptional
section; in this latter case the image of the projection is the cone
$F_{n+1,0}\subset\PP^{n+2}$ (the strict transform of the exceptional
section is blown down to the vertex of the cone).

\item
If $F=F_{0,1}\subset\PP^3$ (that is, if $F\subset\PP^3$ is the smooth
quadric), then the projection $\pi_p$ blows up the point~$p$ and blows
down strict transforms of the two lines passing through~$p$; the image
of the projection is, of course, just the plane.

\item
If $F=F_{n,0}\subset\PP^{n+1}$, where $n>1$ (that is, if $F$ is a cone
over the rational normal curve of degree $n$ in~$\PP^n$), then the
projection $\pi_p$ blows up the point~$p$ and blows down the strict
transform of the generatrix of the cone passing through~$p$; the image
of this projection is the cone $F_{n-1,0}\subset\PP^n$ (if $n=2$,
this cone is just the plane).

\item
Finally, if $F=V\subset \PP^5$ is the Veronese surface, then the
projection $\pi_p$ just blows up the point~$p$, and the image of this
projection is $F_{1,1}\subset\PP^3$.
\end{enumerate}
\end{proposition}

\section{Rational curves with positive self-intersection and algebraic
germs}\label{sec:C.C>0}

\begin{proposition}\label{prop:dim|C|}
Suppose that $X$ is a smooth projective surface and $C\subset X$ is a
curve isomorphic to $\PP^1$. If $(C\cdot C)=m>0$, then the complete
linear system~$|C|$ has no basepoints, $\dim|C|=m+1$, the morphism
$\ph_{|C|}$ is a birational isomorphism between $X$ and
$\ph(X)\subset\PP^{m+1}$, and $\ph(X)$ is a surface in $\PP^{m+1}$ of
minimal degree~$m$.
\end{proposition}

\begin{proof}
Since the normal bundle $\Nc_{X|C}$ is isomorphic to $\Oo_{\PP^1}(m)$,
where $m>0$, one has $h^0(\Nc_{X|C})=m+1$, $h^1(\Nc_{X|C})=0$, so the
Hilbert scheme of the curve $C\subset X$ is smooth and has
dimension~$m+1$ at the point corresponding to~$C$. A general curve
from this $(m+1)$-dimensional family is isomorphic to $\PP^1$; since
$m>0$, through a general point $p\in X$ there passes a
positive-dimensional family of rational curves. Therefore, the
Albanese mapping of $X$ is constant, whence $H^1(X,\Oo_X)=0$. Now it
follows from the exact sequence
\begin{equation}\label{eq:O_C(C)}
0\to\Oo_X\to \Oo_X(C)\xrightarrow\alpha\Oo_C(C)\to 0  
\end{equation}
that the homomorphism $\alpha_*\colon H^0(X,\Oo_X(C))\to
H^0(C,\Oo_C(C))$ is surjective. Since the linear system
$|\Oo_C(C)|=|\Oo_{\PP^1}(m)|$ has no basepoints, the linear system
$|C|$ has no basepoints either, and it follows from~\eqref{eq:O_C(C)}
and the vanishing of $H^1(\Oo_X)$ that $\dim |C|=m+1$. If
$\ph=\ph_{|C|}\colon X\to\PP^{m+1}$ and $Y=\ph(X)$, then $\dim Y=2$
and $\deg Y\cdot\deg \ph=(C\cdot C)=m$. Since $Y\subset \PP^{m+1}$ is
non-degenerate, one has $\deg Y\ge m$ (see~\eqref{eq:deg>codim}), whence $\deg Y=m$ and
$\deg\ph=1$, so $\ph$ is birational onto its image.
\end{proof}

\begin{corollary}\label{cor:dim|C|}
If $X$ is a projective surface with rational singularities and
$C\subset X_{\mathrm{sm}}$ is a curve that is isomorphic to $\PP^1$
and such that $(C\cdot C)>0$, then $h^1(\Oo_X)=0$ and
$h^0(\Oo_X(C))=(C\cdot C)+1$.
\end{corollary}

\begin{proof}
Let $\bar X$ be a desingularization of~$X$. Arguing as in the proof of
Proposition~\ref{prop:dim|C|}, we conclude that through a general
point of $\bar X$ there passes a positive-dimensional family of
rational curves, whence $H^1(\bar X,\Oo_{\bar X})=0$. Since the
singularities of the surface $X$ are rational, $H^1(\bar X,\Oo_{\bar
  X})\cong H^1(X,\Oo_X)$, so $H^1(X,\Oo_X)=0$. Now the result follows
from the exact sequence~\eqref{eq:O_C(C)}.
\end{proof}

Proposition~\ref{prop:dim|C|} implies the following characterisation
of algebraic neighborhoods of rational curves.

\begin{proposition}\label{prop:alg.neighborhoods}
If $(C,U)$ is an algebraic neighborhood of the curve $C\cong \PP^1$ and
if $(C\cdot C)=d>0$, then the germ of this neighborhood is
isomorphic to the germ of neighborhood of a smooth hyperplane section in a
surface of minimal degree~$d$ in~$\PP^{d+1}$.
\end{proposition}

\begin{proof}
Passing to desingularization, one may without loss of generality
assume that the neighborhood in question is a neighborhood of a curve
$C\subset X$, where $C\cong\PP^1$, $X$ is a smooth projective surface,
and $(C\cdot C)=d>0$. If $\ph\colon X\to Y$, where $Y$ is a surface of
minimal degree, is the birational morphism the existence of which is
asserted by Proposition~\ref{prop:dim|C|}, then $\ph$ is an
isomorphism in a neighborhood of $C$ and $\ph(C)$ is a hyperplane
section of $Y$, whence the result.
\end{proof}

Proposition~\ref{prop:alg.neighborhoods} may be regarded as a
generalization of Proposition~4.7 from~\cite{HK}, which asserts that
any algebraic germ of neighborhood of $\PP^1$ with
self-intersection~$1$ is isomorphic to the germ of neighborhood of a
line in~$\PP^2$.

Now we show that not only all germs of algebraic neighborhoods of
$\PP^1$ can be obtained from surfaces of minimal degree, but that
their isomorphisms are induced by isomorphisms of surfaces of minimal
degree.

\begin{proposition}\label{prop:rigid}
Suppose that $X_1\subset\PP^{N_1}$ and $X_2\subset\PP^{N_2}$ are
linearly normal projective surfaces with rational singularities,
$C_1\subset X_1$ and $C_2\subset X_2$ are their smooth hyperplane
sections, and that $C_1$ and $C_2$ are isomorphic to $\PP^1$.

If there exist analytic neighborhoods $U_j\supset C_j$, $j=1,2$, and a
holomorphic isomorphism $\ph\colon U_1\to U_2$ such that
$\ph(C_1)=C_2$, then $\ph$ extends to a projective isomorphism
$\Phi\colon X_1\to X_2$.
\end{proposition}

To prove this proposition we need two lemmas.
The first of them is well known.

\begin{lemma}\label{lemma:const}
If $X$ is a projective surface with isolated singularities
and $C\subset X_{\mathrm{sm}}$ is an ample irreducible curve, then the
ring of germs of holomorphic functions along $C$ coincides with~\C.
\end{lemma}

\begin{proof}[Sketch of proof]
This follows immediately from the fact that $H^0(\hat X,\Oo_{\hat
  X})=\C$, where $\hat X$ is the formal completion of $X$ along~$C$
(see \cite[Chapter V, Proposition 1.1 and Corollary 2.3]{H-ample}).

Here is a more elementary argument. Since $C$ is an ample divisor in
$X$, there exists and embedding of $X$ in $\PP^N$ such that $rC$, for
some $r>0$, is a hyperplane section of $X$. Suppose that $U\supset C$,
$U\subset X$ is a connected neighborhood of~$C$. There exists a family
of hyperplane sections $\{H_\alpha\}$, close to the one corresponding
to $rC$, such that $H_\alpha\subset U$ for each $\alpha$ and
$\bigcup_\alpha H_\alpha$ contains a non-empty open subset $V\subset
U$.  If $f$ is a holomorphic function on $U$, then $f$ is constant on
each $H_\alpha$; since $H_\alpha\cap H_{\alpha'}\ne\varnothing$ for
each $\alpha$ and $\alpha'$, $f$ is constant on the union of all the
$H_\alpha$'s, hence on $V$, hence on~$U$.
\end{proof}

\begin{lemma}\label{lemma:extn2}
Suppose that $X$ is a projective surface such that $H^1(X,\Oo_X)=0$,
$D\subset X_{\mathrm{sm}}$ is an ample irreducible projective curve, and $r$
is a positive integer. Then any germ of meromorphic function
along~$C$, with possibly a pole of order~$\le r$ along~$C$ and no
other poles, is induced by a rational function on~$X$ with possibly a
pole of order~$\le r$ along~$C$ and no other poles.
\end{lemma}

\begin{proof}
Let $\Ic_C\subset\Oo_X$ be the ideal sheaf of $C$; put $\Ic_{rC}=\Ic_C^m$,
$\Oo_{rC}=\Oo_X/\Ic_C^m$, and
\begin{equation*}
  \Nc_{X|rC} = \underline{\mathrm{Hom}}_{\Oo_{rC}}(\Ic_{rC}/\Ic_{rC}^2,\Oo_{rC})=
\underline{\mathrm{Hom}}_{\Oo_X}(\Ic_{rC},\Oo_{rC}).  
\end{equation*}
Identifying $\Oo_X(rC)$ with the sheaf of meromorphic functions having
at worst a pole of order~$\le r$ along~$C$, one has the exact sequence
\begin{equation}\label{eq:O_C(rC)}
0\to\Oo_X\to\Oo_X(rC)\xrightarrow \alpha   \Nc_{X|rC}\to 0,
\end{equation}
in which the homomorphism $\alpha$ has the form
\begin{equation*}
  g\mapsto (s\mapsto gs \bmod \Ic_{rC}),
\end{equation*}
where $g$ is a meromorphic function on an open subset of $V\subset X$,
with at worst a pole of order~$\le r$ along~$C$ and $s$ is a section
of $\Ic_{rC}$ over~$V$. In particular, if $U\supset C$ is a
neighborhood, then any section of $g\in H^0(U,\Oo_X(rC))$ induces a
global section of $\Nc_{X|rC}$. Since $H^1(X,\Oo_X)=0$, the homomorphism $\alpha$
from~\eqref{eq:O_C(rC)} induces a surjection on global sections, so
there exists a section $f\in H^0(X,\Oo_X(rC))$ such that $f$ and $g$
induce the same global section of $\Nc_{X|rC}$. Hence, the
meromorphic function $(f|_U)-g$ has no pole in $U$, so by virtue of
Lemma~\ref{lemma:const} this function is equal to a constant $c$ on a
(possibly smaller) neighborhood of $C$. Thus, the germ of $f-c$
along~$C$ equals that of~$g$.
\end{proof}

\begin{proof}[Proof of Proposition~\ref{prop:rigid}]
It follows from the hypothesis that $(C_1\cdot C_1)=(C_2\cdot C_2)$. If we
denote these intersection indices by~$m$, then
Corollary~\ref{cor:dim|C|} implies that $\dim
|C_1|=\dim |C_2|=m+1$. Let $f_0,\dots,f_{m+1}$ be a basis of
$H^0(\Oo_{X_1}(C_1))$ (i.e., the basis of space of meromorphic
functions on $X_1$ with at worst a simple pole along $C_1$), and
similarly let $(g_0,\dots,g_{m+1})$ be a basis of  $H^0(\Oo_{X_2}(C_2))$.
Embed $X_1$ (resp.~$X_2$) into $\PP^{m+1}$ with the linear
system $|C_1|$ (resp.~$|C_2|$), that is, with the mappings
\begin{equation*}
x\mapsto (f_0(x):\dots:f_{n+1}(x))\quad\text{and}\quad  
y\mapsto (g_0(y):\dots:g_{n+1}(y)).
\end{equation*}
If $\gamma_i$, $0\le i\le m+1$, is the germ along $C_1$ of the
meromorphic function $g_i\circ \ph$, then by virtue of
Lemma~\ref{lemma:extn2}, which we apply in the case $r=1$, each
$\gamma_i$ is the germ along $C_1$ of a meromorphic function $h_i\in
H^0(\Oo_{X_1}(C_1))$. If $h_i=\sum a_{ij}f_j$, then the matrix
$\|a_{ij}\|$ defines a linear automorphism $\Phi\colon\PP^{m+1}\to
\PP^{m+1}$ such that its restriction to a neighborhood of $C_1$
coincides with $\Phi$. Hence, $\Phi$ maps $X_1$ isomorphically
onto~$X_2$.
\end{proof}

\begin{corollary}
Suppose that $F\subset\PP^{m}$
\textup(resp.\ $F'\subset\PP^{m'}$\textup) is a surface of minimal
degree and $C$ \textup(resp.\ $C'$\textup) is its smooth hyperplane
section. Then the germs of neighborhoods of $C$ in $F$ and of $C'$
in~$F'$ are isomorphic if and only if there exists a linear
isomorphism $\Phi\colon \PP^{m}\to\PP^{m'}$ such that $\Phi(F)=F'$ and
$\Phi(C)=C'$.
\end{corollary}

\begin{proof}
Immediate from Proposition~\ref{prop:rigid} if one takes into account
that all surfaces of minimal degree have rational singularities.  
\end{proof}

\begin{note}
We see that any algebraic neighborhood of~$C\cong\PP^1$ has one more
discrete invariant, besides the self-intersection $d=(C\cdot C)$: if
this neighborhood is isomorphic to the germ of neighborhood of the
minimal surface $F_{n,r}$, where $d=n+2r$, this is the integer $n\ge
0$ (and if the surface is not $F_{n,r}$ but the Veronese surface
$V\subset \PP^5$, we assign to our neighborhood the tag~$V$
instead). It should be noted however that, as a rule, the pair $(d,n)$
does \emph{not} determine the germ of neighborhood up to
isomorphism. Indeed, the dimension of the group of automorphisms of the
surface $\Ff_n$ is $n+5$ if $n>0$ and~$6$ if $n=0$. In most cases
this is less that the dimension of the space of hyperplanes in
$\PP^{n+2r+1}$, in which $F_{n,r}$ is embedded.  On the other hand,
linear automorphisms of $F_{n,0}\subset \PP^{n+1}$ act transitively on
the set of smooth hyperplane sections of $F_{n,0}$, and ditto
for~$V\subset\PP^5$. Thus, tags $(n,0)$ or~$V$ do determine an
algebraic germ of neighborhood of~$\PP^1$ up to isomorphism.
\end{note}

\section{Blowups and blowdowns}\label{sec:blowdowns}\label{proof}

Suppose that a smooth projective curve $C$ lies on a smooth
complex analytic surface~$S$ and that $p\in C$. Let $\sigma\colon \tilde
S\to S$ be the blowup of $S$ at $p$, and let $\tilde C\subset \tilde S$ be the
strict transform of $C$. It is clear that the germ of neighborhood
$\tilde C\subset \tilde S$ depends only on the germ of neighborhood
$C\subset S$ and on the point $p\in C$.

\begin{definition}
In the above setting, the germ of neighborhood $(\tilde C,\tilde U)$ will
be called \emph{the blowup} of the germ $(C,U)$ at the point~$p$.
\end{definition}

For future reference we state the following obvious properties of
blowups of neighborhoods.

\begin{proposition}\label{blowup}
Suppose that the germ of neighborhoods $(\tilde C,\tilde U)$ is a blowup
of~$(C,U)$. Then
\begin{enumerate}
\item $\tilde C$ is isomorphic to $C$;  
\item $(\tilde C\cdot \tilde C)=(C\cdot C)-1$;
\item if\label{item4} $(C,U)$ is algebraic then $(\tilde C,\tilde U)$
  is algebraic.
\end{enumerate}
\end{proposition}

For algebraic neighborhoods of $\PP^1$ the assertion~(\ref{item4}) of
Proposition~\ref{blowup} can be made more explicit.

\begin{proposition}\label{prop:projection}
If an algebraic germ $(C,U)$ is isomorphic to the germ of
neighborhood of a hyperplane section of a surface of minimal degree
$F\subset\PP^N$, then the blowup of this germ at a point $p\in C$ is
isomorphic to the germ of neighborhood of a hyperplane section of the
surface $F'\subset\PP^{N-1}$ that is obtained from $F$ by projection
from the point~$p$.
\end{proposition}

\begin{proof}
Immediate from Proposition~\ref{prop:proj2}.  
\end{proof}

\begin{note}
Even though $\PP^1$ is homogeneous, which means that its points are
indistinguishable, germs of blowups of a given neighborhood of $\PP^1$
at different points are not necessarily isomorphic. Indeed, suppose
that $C\cong \PP^1$ is a smooth hyperplane section of the
surface~$F_{n,r}$, where $n>0$ and $r>0$. If $p\in C$ does not lie on
the exceptional section $E\subset F_{n,r}$, then
Proposition~\ref{prop:proj} implies that the blowup at $p$ of the germ
of neighborhood of~$C$ is isomorphic to a neighborhood of a hyperplane
section of $F_{n-1,r}$, and if $p$ does lie on $E$,
Proposition~\ref{prop:proj} implies that the blowup in question is
isomorphic to a neighborhood of a hyperplane section of $F_{n+1,r-1}$
(observe that $C\ne E$ and $C\cap E\ne\varnothing$, so points of both
kinds are present). If the blowups at such points were isomorphic as
germs of neighborhoods, then, by virtue of
Proposition~\ref{prop:rigid}, this isomorphism would be induced by a
linear isomorphism between $F_{n-1,r}$ and $F_{n+1,r-1}$, which does
not exist.
\end{note}

\begin{proposition}\label{non-alg}
Suppose that a curve $D\cong\PP^1$ is embedded in a surface $U$. If,
blowing up $s>0$ points of the germ of neighborhood $(D,U)$, one
obtains a germ of neighborhood $(C,W)$ that is isomorphic to the germ
of neighborhood of a non-degenerate conic in $\PP^2$, then the
original germ $(D,U)$ is not algebraic.
\end{proposition}

\begin{proof}
Without loss of generality we may and will assume that $C$ is a conic
in $\PP^2$ and $W\supset C$ is an open subset of $\PP^2$. The
Veronese mapping~$v\colon\PP^2\to V$ (see~\eqref{eq:Veronese})
identifies $\PP^2$ with the Veronese surface $V\subset\PP^5$ and
$C\subset \PP^2$ with a smooth hyperplane section of~$V$.

Arguing by contradiction, suppose that the germ $(D,U)$ is algebraic. Then
Proposition~\ref{prop:alg.neighborhoods} implies that this
germ is isomorphic to the germ of neighborhood of a
hyperplane section of a surface of minimal degree~$m=4+s>4$. Hence,
this surface is of the form $F_{n,r}$, where $n+2r=m$ (see
Proposition~\ref{prop:min.degree}).

By construction, the germ of~$(C,W)$ can be obtained from the
germ $(D,U)$ by blowing up $s$ points on~$D$. Hence, by
Proposition~\ref{prop:projection}, the germ of $(C,W)$ is isomorphic
to the germ of a hyperplane section of a surface that can be obtained
from $F_{n,r}$ by $s$ consecutive projections. However,
Proposition~\ref{prop:proj} shows that the resulting surface cannot be
projectively isomorphic to~$V\subset\PP^5$
(cf.\ Figure~\ref{fig:hierarchy}). On the other hand,
Proposition~\ref{prop:rigid} implies that if germs of hyperplane
sections of two surfaces of minimal degree are isomorphic then the
surfaces are projectively isomorphic.  This contradiction shows that
the germ~$(D,U)$ is not algebraic.
\end{proof}

Now we can construct our first series of non-algebraic examples.

\begin{lemma}\label{lemma:tubular}
Suppose that $C\subset \PP^2$ is a non-degenerate conic and $s$ is a
positive integer. Then there exist $s$ lines $\lst Ls\subset\PP^2$ and
neighborhoods $W\supset C$, $W_j\supset L_j$, $1\le j\le s$, in $\PP^2$
having the following properties.

\begin{enumerate}
\item each $L_j$ intersects the conic $C$ at precisely two points,
  $p_j$ and~$q_j$;
\item all the points \lst ps are distinct, all the points \lst qs
  are distinct, and $p_i\ne q_j$ for any $i,j$.
\item\label{e3}
$W_j\cap W\cap\{\lst ps,\lst qs\}=\{p_j,q_j\}$ for each $j$.
\item\label{e4}
The open subset $W_j\cap W$ has precisely two connected components
$P_j\ni p_j$ and $Q_j\ni q_j$.
\item ${P_i}\cap {P_j}=\varnothing$\label{new:item3} whenever $i\ne
  j$, ${Q_i}\cap {Q_j}=\varnothing$ whenever $i\ne j$, and ${P_i}\cap
  {Q_j}=\varnothing$ for any $i,j$.
\end{enumerate}
\end{lemma}

\begin{proof}
Only the assertions~(\ref{e3}) and (\ref{e4}) deserve a sketch of
proof. To justify them choose a Hermitian metric on $\PP^2$ and let
$W$ be a small enough tubular neighborhood of $C$ and \lst Ws be small
enough tubular neighborhoods of \lst Ls.
\end{proof}

The proof of the following lemma is left to the reader.

\begin{lemma}\label{NT}
Suppose that $X_1$ and $X_2$ are Hausdorff topological spaces,
$O_1\subset X_1$ and $O_2\subset X_2$ are open subsets, and $X$ is the
topological space obtained by gluing $X_1$ and $X_2$ via a
homeomorphism $\Phi\colon O_1\to O_2$.

If for any $x_1\in \bd(O_1)\subset X_1$, $x_2\in \bd(O_2)\subset X_2$,
where $\bd$ means ``boundary'',
there exist open neighborhoods $V_1\ni x_1$ in~$X_1$ and $V_2\ni x_2$
in $X_2$ such that 
\begin{equation*}
\Phi(V_1\cap O_1)\cap V_2=\varnothing,
\end{equation*}
then $X$ is Hausdorff.\qed
\end{lemma}

\begin{constr}
Suppose that $C$, \lst Ls, \lst Ws, and $W$ are as in
Lemma~\ref{lemma:tubular}.  

Put 
\begin{equation*}
U_0=W\sqcup W_1\dots\sqcup W_s  
\end{equation*}
(disjoint sum), and let $\pi_0\colon U_0\to\PP^2$ be the natural projection.
Define the equivalence relation~$\sim$ on~$U_0$ as follows: if $x,y\in
U_0$ and $x\ne y$, then $x\sim y$ if and only if  $\pi_0(x)=\pi_0(y)\in
P_1\cup\dots\cup P_s$. 

Let $U_1$ be the quotient of $U_0$ by the equivalence relation~$\sim$,
and let $\pi_1\colon U_1\to\PP^2$ be the natural projection. 

Denote the images of $C\subset W$ and $L_j\subset W_j$ in $U_1$ by
$C_1$ and $L_j'$.
\end{constr}

\begin{lemma}
In the above setting, $U_1$ is a Hausdorff
and connected complex surface and $\pi_1\colon
U_1\to\PP^2$ is a local holomorphic isomorphism.

The curves $C_1$ and $L_j'$ are isomorphic to $\PP^1$,
$(L'_j\cdot L'_j)=(L_j\cdot L_j)=1$, $(L'_j\cdot C_1)=1$ for
each~$j$, and $(C_1\cdot C_1)=(C\cdot C)=4$.
\end{lemma}

\begin{proof}
If we put, in Lemma~\ref{NT}, $X_1=W$, $X_2=W_1\sqcup\dots\sqcup W_s$,
$O_1=P_1\cup\ldots \cup P_s\subset W$, $O_2=P_1\sqcup \dots\sqcup
P_s\subset W_1\sqcup\dots\sqcup W_s$, $\Phi=\mathrm{Id}$, then the
hypothesis of this lemma is satisfied if, putting
$P=P_1\cup\dots\cup P_s$ and $Q=Q_1\cup\dots \cup Q_s$, one has
\begin{equation}\label{eq:detailed}
((\bar P\cap W)\setminus P)\cap (\bar P\cap (W_1\cup\dots\cup
  W_s)\setminus P)=\varnothing.
\end{equation}
The left-hand side of \eqref{eq:detailed} is equal to
\begin{equation*}
	(\bar P\cap W\cap (W_1\cup\dots\cup
	W_s))\setminus P=(\bar P\setminus P)\cap (P\cup Q)\subset\bar P\cap Q,
\end{equation*}
which is empty by Lemma~\ref{lemma:tubular}(\ref{new:item3}). Hence,
$U_1$ is Hausdorff.

The rest is obvious.
\end{proof}

\begin{constr}\label{constr1}
In the above setting, for each $j$, $1\le j\le s$, choose two
distinct points $u_j,v_j\in L'_j$, different from the intersection
point of $L'_j$ and $C_1$. Let $U_2$ be the blowup of the surface
$U_1$ at the points $\lst us,\lst vs$. If $\sigma\colon U_2\to U_1$ is
the natural morphism, put $C_2=\sigma^{-1}(C')\subset U_2$; it is
clear that $\sigma$ is an isomorphism on a neighborhood of $C_2$; in
particular, $C_2\cong C_1\cong \PP^1$ and $(C_2\cdot C_2)=4$.

For each $j$, let $\tilde L_j\subset U_2$ be the strict transform of
$L'_j$ with respect to the blowup $\sigma\colon U_2\to U_1$. By
construction, $\tilde L_j\cong\PP^1$, $(\tilde L_j\cdot\tilde L_j)=-1$
for each $j$, and the curves $\tilde L_j$ are pairwise
disjoint. Hence, one can blow down the curves $\lst{\tilde L}s$ to
obtain a smooth complex surface $U_3$ and a curve $C_3\subset U_3$,
which is the image of~$C_2$; one has $C_3\cong \PP^1$ and $(C_3\cdot
C_3)=(C_2\cdot C_2)+s=4+s>4$.  
\end{constr}

\begin{proposition}\label{prop:constr1}
If $(C_3,U_3)$ is the germ of neighborhood from
Construction~\ref{constr1}, then this germ is not algebraic and
$\trdeg_\C\M(U_3,C_3)=2$.   
\end{proposition}

\begin{proof}
It follows from the construction that the blowup of the germ
of neighborhood of $C_3$ in $U_3$ at the $s$ points to which
\lst{\tilde L}s were blown down, is isomorphic to the germ of
neighborhood of $C_2$ in $U_2$, which is isomorphic to the of
neighborhood of the conic $C$ in $\PP^2$. Now
Proposition~\ref{non-alg} implies that the germ $(C_3,U_3)$ is not
algebraic.

Since $\pi_1\colon U_1\to\PP^2$ is a local isomorphism, the filed of
meromorphic functions on $\PP^2$, which is isomorphic to the field
$\C(X,Y)$ of rational functions in two variables, can be embedded in
the field of meromorphic functions on $U_1$. Since the surface $U_3$
is obtained from $U_1$ by a sequence of blowups and blowdowns, the
fields of meromorphic functions on $U_1$ and $U_3$ are
isomorphic. Hence, $\C(X,Y)$ can be embedded in the field of
meromorphic functions on $U_3$, which can be embedded
in~$\M(U_3,C_3)$. Thus, $\trdeg_\C\M(U_3,C_3)\ge2$, whence
$\trdeg_\C\M(U_3,C_3)=2$. This completes the proof.
\end{proof}

\section{Ramified coverings}\label{sec:ramified}

In this section we construct another series of examples of
non-algebraic neighborhoods $(C,U)$, where $C\cong \PP^1$ and
$\trdeg_\C\M(U,C)=2$. In these examples, the
self-intersection $(C\cdot C)$ may be an arbitrary positive
integer. We begin with two simple lemmas.

\begin{lemma}\label{pi1}
Suppose that $X$ is a smooth complex surface, $C\subset X$ is a
projective curve, $C\cong\PP^1$, and $U\subset X$ is a tubular
neighborhood of $C$. Then $\pi_1(U\setminus C)\cong\Z/m\Z$, where
$m=(C\cdot C)$.
\end{lemma}

\begin{proof}
Immediate from the homotopy exact sequence of the fiber bundle
$U\setminus C\to C$.  
\end{proof}

\begin{lemma}\label{lemma:covering}
Suppose that $f\colon S\to T$ is a dominant morphism of smooth
projective algebraic surfaces and that $B\subset T$ is the
branch divisor of~$f$ (see the definition in
Section~\ref{subsec}). If $T\setminus B$ is simply connected, then
$\deg f=1$.  
\end{lemma}

\begin{proof}
The critical locus $R\subset T$ of $f$ is of the form $B\cup E$, where $B$ is the branch
divisor and $E$ is a finite set. The mapping
\begin{equation*}
f|_{S\setminus f^{-1}(R)}\colon S\setminus f^{-1}(R)\to T\setminus R
\end{equation*}
is a topological covering. Since the subset $E\subset T\setminus B$ is
finite and $T\setminus B$ is smooth, fundamental groups of $T\setminus
R$ and $T\setminus B$ are isomorphic, so $T\setminus
R$ is also simply connected, whence the result.
\end{proof}

\begin{constr}\label{constr2}
Fix an integer $n>0$. We are going to construct a certain neighborhood
of $\PP^1$ with self-intersection~$n$.

To that end, suppose that $X\subset\PP^{2n+1}$ is a non-degenerate
surface of degree~$2n$ which is not the cone $F_{2n,0}$ and, if $n=2$,
not the Veronese surface~$V$. Let $C\subset X$ be a smooth hyperplane
section, and let $U\supset C$, $U\subset X$ be a tubular neighborhood
of~$C$. By virtue of Lemma~\ref{pi1} one has $\pi_1(U\setminus
C)=\Z/2n\Z$. Hence, there exists a two-sheeted ramified covering
$\pi\colon V\to U$ that is ramified along~$C$ with index~$2$. If
$C'=\pi^{-1}(C)$, then $C'\cong\PP^1$ and $(C'\cdot C')=n$.
\end{constr}

\begin{proposition}\label{prop:constr2}
If $(C',V)$ is the neighborhood from Construction~\ref{constr2}, then
$\trdeg_\C\M(V,C')=2$ and  
the neighborhood $(C',V)$ is not algebraic.  
\end{proposition}

\begin{proof}
The neighborhood $(C,U)$ is algebraic, so $\trdeg_\C\M(U,C)=2$, and
the morphism $\pi\colon V\to U$ induces an embedding of $\M(U,C)$ in
$\M(V,C')$, hence $\trdeg_\C\M(V,C')\ge2$, hence this transcendence
degree equals~$2$.

To prove the non-algebraicity of $(C',V)$, assume the converse.  Then,
by virtue of Proposition~\ref{prop:alg.neighborhoods}, the germ of the
neighborhood $(C',V)$ is isomorphic to the germ of neighborhood of a
smooth hyperplane section of a non-degenerate surface
$X'\subset\PP^{n+1}$, $\deg X'=n$; we will identify $C'$ with this
hyperplane section and $V$ with a neighborhood of $C'$ in $X'$.

Let \lst [0]f{2n} be a basis of the space $H^0(X,\Oo_X(C))$ (i.e., of
the space of meromorphic functions on $X$ with at worst a simple pole
along~$C$). For each $j$, the function $(f_j|_U)\circ\pi$ is a
meromorphic function on $V$ with at worst a pole of order~$2$
along~$C'$. Using Lemma~\ref{lemma:extn2} (in which one puts $r=2$),
one sees that there exist meromorphic functions $\lst[0]g{2n}\in
H^0(X',\Oo_{X'}(2C'))$ such that, for each~$j$, the germ of $g_j$
along~$C$ is the same as that of $(f_j|_U)\circ\pi$. 

Choose a basis \lst[0]hN is of $H^0(X',\Oo_{X'}(2C'))$, and let
$X_1\subset \PP^N$ be the image of $X'$ under the embedding
\begin{equation*}
x\mapsto (h_0(x):\dots:h_N(x))  
\end{equation*}
(this is the embedding defined by the complete linear system
$|2C|=|\Oo_{X'}(2)|$). 
If $g_j=\sum a_{jk}h_k$ for each $j$, $0\le
j\le 2n$, then the matrix $\|a_{jk}\|$ defines a rational mapping
$p\colon X_1\dasharrow X$, induced by a linear projection
$\bar p\colon \PP^N\dasharrow \PP^{2n}$.  Hence,
\begin{equation}\label{eq:ineq.deg}
\deg X\le\frac{\deg X_1}{\deg p},
\end{equation}
and the equality is attained if and only if the rational mapping $p$
is regular.

On the other hand, $\deg X_1=4\deg X'=4n$, $\deg X=2n$, and $\deg p\ge
2$ since the restriction of $p$ to~$V\subset X'\cong X_1$ coincides
with our ramified covering~$\pi\colon V\to U$. Now it follows
from~\eqref{eq:ineq.deg} that the projection $p$ is regular (i.e., its
center does not intersect $X_1$) and $\deg p=2$.

If $X'\subset\PP^{n+1}$ is not a cone (i.e., if $X'\ne F_{n,0}$), put
$X_2=X_1$ and $q=p$, and if $X'$ is the cone~$F_{n,0}$, put $X_2=\Ff_n$ and
$q=p\circ \sigma$, where $\sigma\colon \Ff_n\to F_{n,0}=X'$ is the
standard resolution. So, we have a holomorphic mapping $q\colon X_2\to
X$, $\deg q=2$. One has either $X_2\cong\Ff_n$, or $X_2\cong\PP^2$ (the
latter case is possible only if $n=4$ and $X'\subset\PP^5$ is the
Veronese surface).

Since $q$ agrees with $\pi$ on a neighborhood of the curve $C'\subset
X_2$, the curve $C\subset X$ is contained in the branch
divisor of~$q$. Let us show that the branch divisor of $q\colon X_2\to X$
coincides with~$C$.

To that end, denote this branch divisor by $B\subset X$. Let $D\subset
X$ be a general hyperplane section, and put $D_2=q^{-1}(D)\subset X_2$. For a
general $D$, one has $D\cong\PP^1$, $D_2$ is a smooth and connected
projective curve, and the morphism $q|_{D_2}\colon D_2\to\PP^1\cong D$
is ramified over $\deg B$ points.

If $n=4$ and $X'$ is the Veronese surface $V\subset\PP^5$, then
$(X_2,\Oo_{X_2}(D))\cong (\PP^2,\Oo_{\PP^2}(4))$, so the curve $D_2$
is isomorphic to a smooth plane quartic. Such a curve does not admit a
mapping to $\PP^1$ of degree~$2$, so this case is impossible. 

Thus, $X'=F_{k,l}$, where $k+2l=n$. In the notation of
Section~\ref{recap}, the divisor $D_2\subset X_2$ is equivalent to
$2(e+(k+l)f)$; since the canonical class~$K_{X_2}$ of the surface
$X_2\cong\Ff_k$ is equivalent to $-2e-(k+2)f$, one has, denoting the
genus of $D_2$ by~$g$,
\begin{equation*}
2g-2=(D_2\cdot D_2+K_{X_2})=(2e+2(k+l)f\cdot (k+2l-2)f)=2(n-2),  
\end{equation*}
whence $g=n-1$. Applying Riemann--Hurwitz formula to the degree~$2$
morphism $q|_{D_2}\colon D_2\to D$, one sees that the number of its
branch points equals~$2n$. So,
$\deg B=2n$. Since $B\supset C$ and $\deg C=2n$, one
has $B=C$.

Observe now that $X\setminus C$ is simply connected since $X\setminus
C$, as a topological space, is a fiber bundle over $\PP^1$ with the
fiber~\C, and the complement $X_2\setminus q^{-1}(C)$ is
connected. Applying Lemma~\ref{lemma:covering} to the mapping $q\colon
X_2\to X$, one sees that $\deg q=1$.  We arrived at a contradiction.
\end{proof}

\bibliographystyle{amsplain}
\bibliography{ample}

\end{document}